\documentclass[12pt,reqno]{amsart}

\usepackage{amsfonts}
\usepackage{amsmath}
\usepackage{amssymb}
\usepackage[colorlinks = True, linkcolor=blue, urlcolor=blue]{hyperref}
\usepackage{thmtools}
\usepackage{amsthm}
\usepackage{verbatim}
\usepackage{ulem}
\usepackage{cancel}
\usepackage{tikz, tikz-cd}
\usetikzlibrary{calc}
\usepackage{ytableau}
\usepackage{subcaption}

\theoremstyle{plain}
\newtheorem{theorem}{Theorem}[section]
\newtheorem{proposition}[theorem]{Proposition}
\newtheorem{lemma}[theorem]{Lemma}
\newtheorem{corollary}[theorem]{Corollary}

\theoremstyle{definition}

\theoremstyle{remark}

\theoremstyle{example}

\numberwithin{equation}{section}

\usepackage[]{cite}

\newcommand\QQ{\mathbb{Q}}

\newcommand\ZZ{\mathbb{Z}}

\newcommand\Par{\mathrm{Par}}

\newcommand\up{\mathrm{up}}

\newcommand\inv{\mathrm{inv}}

\newcommand\arm{\mathrm{arm}}
\newcommand\leg{\mathrm{leg}}
\newcommand{\coleg}{\mathrm{coleg}}
\newcommand{\coarm}{\mathrm{coarm}}

\newcommand{\rev}{\mathrm{rev}}

\newcommand{\SYT}{\mathrm{SYT}}

\newcommand{\wt}{\mathrm{wt}}
\newcommand{\Area}{\mathrm{Area}}

\newcommand{\DD}{\mathbb{D}}

\newcommand{\NEpath}[4]{
	\fill[white!25]  (#1) rectangle +(#2,#3);
	\fill[fill=white]
	(#1)
	\foreach \dir in {#4}{
		\ifnum\dir=0
		-- ++(1,0)
		\else
		-- ++(0,1)
		\fi
	} |- (#1);
	\draw[help lines] (#1) grid +(#2,#3);
	\draw[dashed] (#1) -- +(#3,#3);
	\foreach \x in {1,...,#2} {
		\node[below,gray] at ($(#1) + (\x-0.5, -0.3)$) {\x};
	}
	
	\foreach \y in {1,...,#3} {
		\node[left,gray] at ($(#1) + (-0.3, \y-0.5)$) {\y};
	}
	\coordinate (prev) at (#1);
	\foreach \dir in {#4}{
		\ifnum\dir=0
		\coordinate (dep) at (1,0);
		\else
		\coordinate (dep) at (0,1);
		\fi
		\draw[line width=2pt,-stealth] (prev) -- ++(dep) coordinate (prev);
	};
}

\title[$q$-rook numbers]{A tableaux formula for $q$-rook numbers}
\author[Basu]{Tirtharaj Basu}
\address{The Institute of Mathematical Sciences, A CI of Homi Bhabha National Institute, Chennai 600113, India}
\email{tirtharajb@imsc.res.in}
\author[Bhattacharya]{Aritra Bhattacharya}
\address{Beijing International Center for Mathematical Research, Peking University, Beijing 100871, China}
\email{matharitra@gmail.com}

\date{\today}
\begin{document}
	
	\maketitle

\begin{abstract}
	We provide a formula for the Garsia-Remmel $q$-rook numbers as a sum over standard Young tableaux. We connect our formula with the coefficients in $q$-Whittaker expansion of unicellular LLT functions. 
\end{abstract}

\section{Introduction}

The Garsia-Remmel $q$-rook numbers $R_k(\lambda;q) \in \ZZ_{\geq 0}[q]$ for $k \in \ZZ_{\geq 0}$ counts the number of ways to place $k$ non-attacking rooks on a Ferrers board of a partition $\lambda$ with certain $q$ weight. We provide a tableaux formula for $R_k(\lambda;q)$, which we describe now.

Let $\pi$ be a Dyck path of semilength $n$ and let $\lambda(\pi)$ denote the partition formed by the shape above $\pi$ inside the $n \times n$ grid. For $1\leq i<j \leq n$ let $i<_\pi j$ if $(i,j) \notin \Area(\pi)$, i.e, the cell $(i,j)$ is above the Dyck path $\pi$. The set $\SYT^\pi_\mu$ is the set of standard Young tableaux of shape $\mu$ such that if  $i$ is above $j$ in the same column then $i <_\pi j$. Let \begin{align*}
	\gamma(T) = \# \{(b,c) \in \mu\times \mu\,|\, \coleg(b)>\coleg(c) \hbox{ and } (T(c),T(b)) \in \Area(\pi) \}
\end{align*} be the number of pairs of boxes $(b,c)$ such that $c$ is in some row above $b$ in the Young diagram (English notation) and $T(c) < T(b)$ but $T(c) \nless_\pi T(b)$. 

We can now state our main result. 
 \begin{theorem}\label{th:mainthmintrostate}
	Let $n \in \ZZ_{>0}$, $\lambda \in \Par$ and $\pi \in \DD_n$ is such that $\lambda(\pi) = \lambda$. Then for $k \in \ZZ_{\geq 0},$ \begin{equation}\label{eq:mainthmeq}
		R_k(\lambda;q) = \sum_{\substack{\mu \vdash n \\ \mu_1 = n-k}} q^{n(\mu')-\# \Area(\pi)} \sum_{T \in \SYT^\pi_\mu}
		q^{\gamma(T)} \prod_{\substack{b \in \mu \\ \coleg(b)>0}} [\arm_{<_\pi T(b)}(\up(b))+1]_q.
	\end{equation} 
\end{theorem}
A more detailed explanation of all the notations used above is given in  \S\ref{sec:tabformulanots}.
 
In fact, for $n = N \geq \lambda_1 + \lambda'_1$, the above formula only runs over the partition $(N-k,k)$ and so (\autoref{prop:Rkabelian}) $$ R_k(\lambda;q) = q^{|\lambda|-(N-k)k} \sum_{T \in \SYT^\pi_{(N-k,k)}} q^{\gamma(T)} \prod_{\substack{b \in \mu \\ \coleg(b)>0}} [\arm_{<_\pi T(b)}(\up(b))+1]_q. $$

In \cite{Abreu_Nigro_csfmodular}, \cite{CMP_rook} and \cite{Ram_Schlosser_qWandrook}  relations between $q$-rook numbers and symmetric functions appearing in the Macdonald functions universe are explored. We use the formula above to make yet another such connection.
By using result of \cite{GMRWW_Macexpansion}, we can relate our formula to the coefficients of unicellular LLT functions $\chi_\pi(q)$ for a Dyck path $\pi$ in the basis of $q$-Whittaker functions $(W_\lambda(q):\lambda \in \Par)$. For $\pi \in \DD_n$ and partitions $\mu \vdash n$, let $c_{\pi,\mu}(q) \in \QQ(q)$ be defined by \begin{align*}
	\chi_\pi(q) = \sum_{\mu \vdash n} (1-q)^{n-\mu_1} c_{\pi,\mu}(q) W_\mu(q).
\end{align*} 
Then (\autoref{prop:Rkwithcpimu})
\begin{align*}
	\sum_{\substack{\mu \vdash n \\ \mu_1 = n-k}} q^{n(\mu')-\# \Area(\pi)} c_{\pi,\mu}(q) = R_k(\lambda(\pi);q).
\end{align*}
The recent paper \cite{kim2025halllittlewoodexpansionschromaticquasisymmetric} obtains another proof of the above identity.

Based on the formula for $e$-expansion for unicellular LLTs obtained in \cite{Abreu_Nigro_forests}, we also connect the last $q$-rook number of certain partitions to the $e$-expansion coefficients in \autoref{prop:enrook}.

\section{Notations}

\subsection{}We denote by $[n]$ the integer interval $\{1,\ldots,n\}$ for $n \in \ZZ_{>0}$. This is not to be confused with the $q$-numbers $[n]_q$ which will always have a $q$ in the subscript, and also should be clear from the context. Unless otherwise mentioned, $n$ is some positive integer in this paper.

\subsection{q-numbers} For $n,k \in \ZZ_{\geq 0}$ with $0 \leq k \leq n$, \begin{align*}
	[n]_q = 1+\ldots+q^{n-1}\quad \hbox{ and } \quad [n]_q! = [n]_q\ldots[1]_q.
\end{align*}
Let $(a;q)_j = (1-a)(1-qa)\ldots(1-q^{j-1}a)$, for $j \in \ZZ_{\geq 0}$. Then \begin{equation}\label{eq:qbindef}
	{n\brack k}_q = \dfrac{(q;q)_{n}}{(q;q)_k(q;q)_{n-k}} = \dfrac{[n]_q!}{[k]_q! [n-k]_q!}.
\end{equation} 
Then \begin{align}\label{eq:qbinomialreverse}
	{j \brack k}_{q^{-1}} &= \dfrac{(q^{-1};q^{-1})_j}{(q^{-1};q^{-1})_k (q^{-1};q^{-1})_{j-k}} 
	=q^{\binom{k}{2}+\binom{j-k}{2}-\binom{j}{2}} {j \brack k}_q = q^{-k(j-k)} {j \brack k}_q.
\end{align}

\subsection{Dyck paths} A Dyck path of semilength $n$ is a lattice path from $(0,0)$ to $(n,n)$ consisting of unit length north steps $N$ and unit length east steps $E$ such that the path always stays weakly above the diagonal $x=y$. We will write a Dyck path as a word in $N$ and $E$. We write the cell co-ordinate of each box in the $n\times n$ grid from $(0,0)$ to $(n,n)$ inside $\ZZ_{\geq 0} \times \ZZ_{\geq 0}$ as the co-ordinate of its north-east corner. $\Area(\pi)$ is the set of cells below $\pi$ above the diagonal. The set of Dyck paths of semilength $n$ is denoted by $\DD_n$, and $\DD = \cup_{n \in \ZZ_{\geq 0}} \DD_n$ be the set of all Dyck paths of any semilength. For $\pi \in \DD$, let $|\pi|$ denote its semilength. \autoref{fig:pi} gives an example.

\begin{figure}[h]
	\begin{minipage}{0.45 \textwidth}
		\begin{tikzpicture}
			\NEpath{0,0}{6}{6}{1,1,1,0,0,0,1,1,0,1,0,0};
		\end{tikzpicture}
	\end{minipage}
	\hfill
	\begin{minipage}{0.45 \textwidth}
		\begin{align*}
			&\pi = N^3E^3N^2ENE^2 \in \DD_6,
			\\
			&\Area(\pi) = \{(1,2),(1,3),(2,3),(4,5),(5,6)\},
			\\
			&\lambda(\pi) = (4,3,3). 	
		\end{align*}
	\end{minipage}
		\caption{Example of a Dyck path}
	\label{fig:pi}
\end{figure}

\subsection{Dyck path to poset}

For a Dyck path $\pi$ of semilength $n$, define a poset on $[n]$ where the strict inequalities are given by $i<_\pi j$ if $i<j$ and $(i,j) \notin \Area(\pi)$, i.e, $(i,j)$ is above the Dyck path.

For the Dyck path in \autoref{fig:pi}, $1 <_\pi 4 <_\pi 6$, $1 <_\pi 5$, $2 <_\pi 4$, $2 <_\pi 5$, and $3<_\pi 4$, $3 <_\pi 5$. Note that if $1 \leq i <_\pi j \leq k \leq n$ implies that $i<_\pi k$.

\subsection{Partitions}The set of all integer partitions is denoted by $\Par$. We think of the Young diagram in the English convention, as in Macdonald's book \cite{MacMainBook}, and follow Macdonald's definition and convention throughout the paper concerning partitions. In particular, for a partition $\lambda$, its conjugate is denoted $\lambda'$, the weighted size $$ n(\lambda') = \sum_{i \geq 1} \binom{\lambda_i}{2}.$$  The $\arm, \leg, \coarm, \coleg$ of a box in the Young diagram is denoted $a,l,a',l'$ respectively in \cite{MacMainBook}. In particular, the cell co-ordinates of a box $b$ equals $(\coleg(b)+1,\coarm(b)+1)$ and $$\arm(b)+\coarm(b)+1 = \lambda_{\coleg(b)+1}, \qquad \hbox{ and } \qquad \leg(b)+\coleg(b)+1 = \lambda'_{\coarm(b)+1}.$$ 

For a box $b \in \lambda$, we denote by $\up(b)$ the box directly above it in the previous row, if $\coleg(b)>0$. So, $$ \coarm(\up(b)) = \coarm(b) \qquad \hbox{ and } \qquad \coleg(\up(b))+1 = \coleg(b). $$ 

Let $\lambda \in \Par$. Denote by $\lambda \pm \varepsilon_i = (\lambda_1,\ldots,\lambda_{i-1},\lambda_i \pm 1, \lambda_{i+1},\ldots)$ the composition obtained by adding or removing a box in the $i$th row of $\lambda$. If $\lambda_{i-1}>\lambda_i$ then $\lambda+\varepsilon_i \in \Par$ and if $\lambda_i > \lambda_{i+1}$ then $\lambda- \varepsilon_i \in \Par$.

\subsection{Dyck path to partition}

The boxes in $n\times n$ grid above $\pi \in \DD_n$ is the shape of a partition, read row-by-row from top to bottom, which we denote by $\lambda(\pi)$. It is contained inside the staircase shape partition $\rho_n = (n-1,\ldots,0)$. \autoref{fig:pi} gives an example.

Then $i<_\pi j$ if $j > n - \lambda(\pi)'_i$.

\subsection{$\pi$-tableaux}
The set of standard Young tableaux of some partition shape $\lambda$ will be denoted by $\SYT_\lambda$. This is the set of fillings $T : \lambda \to [|\lambda|]$ such that the value increases left-to-right along a row and top-to-bottom along a column.

 For $\pi \in \DD_n$ and $\mu \vdash n$ let \begin{equation}
	\SYT_\mu^\pi = \{T \in \SYT_\mu\,|\, T(\up(b)) <_\pi T(b) \ \hbox{ for all } b \in \mu \hbox{ with } \coleg(b) >0 \},
\end{equation} i.e, $\SYT^\pi_\mu$ is the set of standard Young tableaux of shape $\mu$ such that if the number $i$ appears above $j$ in the same column then $i <_\pi j$.

For the path in \autoref{fig:pi}, \begin{equation}\label{eq:SYTpieg}
	\ytableausetup{centertableaux, nosmalltableaux} \SYT^\pi_{(3,3)} = \bigg\{\,\ytableaushort{123,456}\,\bigg\}, \quad \SYT^\pi_{(3,2,1)} =\bigg\{\,\ytableaushort{123,45,6}\,\bigg\}, \quad \hbox{ and } \quad \SYT^\pi_{(3,1^3)} = \emptyset.
\end{equation} 

\subsection{Dyck path to Hessenberg functions}\label{sec:DycktoHess}
Let $n \in \ZZ_{>0}$. A Hessenberg function $\mathbf{m}: [n] \to [n]$ is a non-decreasing function such that $\mathbf{m}(i) \geq i$ for every $i \in [n]$. For a Dyck path $\pi \in \DD_n$, define a Hessenberg function $\mathbf{m}(\pi) : [n] \to [n]$ by $$ \mathbf{m}(\pi)(i) = n-\lambda(\pi)_i, \qquad \hbox{ for } i \in [n],$$ i.e, the value of $i$ is the distance between the $(n-i)$th $E$ step and the line $y=n$, or in other words, $\mathbf{m}(\pi)$ in reverse is the complementary partition of $\lambda(\pi)$ in the $n\times n$ square. 

For the path in \autoref{fig:pi}, $\mathbf{m}(\pi) = (2,3,3,6,6,6)$, where the $i$th component denotes the value at $i$.

\section{$q$ rook numbers}

In this section, we recall the definition and recursion of $q$-rook numbers as defined by Garsia and Remmel in \cite{Garsia_Remmel_rook}. We then provide a proof of our main result, a standard tableaux formula for the $q$-rook numbers.

\subsection{Definition of $q$-rook numbers}Given a partition $\lambda$ and $k \in \ZZ_{\geq 0}$, a \textit{rook placement} with $k$ rooks on $\lambda$ is the the number of ways to select $k$ cells called \textit{rooks} from the Young diagram of $\lambda$, such that no two rooks lie in the same row or column. Denote the set of rook placements with $k$ rooks on $\lambda$ by $\mathcal{C}_k(\lambda)$. Given such a rook placement $C \in \mathcal{C}_k(\lambda)$, \cite{Garsia_Remmel_rook} defines $\inv(C)$ to be the number of cells remaining after cancelling all the cells in the same column above and in the same row to the left of the rooks. Then \begin{equation}
	R_k(\lambda;q) = \sum_{C \in \mathcal{C}_k(\lambda)} q^{\inv(C)}.
\end{equation} \autoref{fig:rook-placement-cancelled} gives an example of a rook placement with the $\inv$ statistics.
\begin{figure}[h]
	\centering
	\begin{tikzpicture}[scale=0.8]

		\foreach \x in {0,...,5} {\draw (\x,0) rectangle ++(1,-1);}     
		\foreach \x in {0,...,3} {\draw (\x,-1) rectangle ++(1,-1);}    
		\foreach \x in {0,...,3} {\draw (\x,-2) rectangle ++(1,-1);}    
		\foreach \x in {0,1}     {\draw (\x,-3) rectangle ++(1,-1);}    
		\draw (0,-4) rectangle ++(1,-1);                                

		\node at (1.5,-1.5) {$\bullet$};
		\node at (0.5,-3.5) {$\bullet$}; 
		\node at (3.5,-2.5) {$\bullet$};
		\draw[thick, red] (1.5,0) -- (1.5,-1.5);
		\draw[thick, red] (0,-1.5) -- (1.5,-1.5);
		\draw[thick, red] (0.5,0) -- (0.5,-3.5);
		\draw[thick, red] (0,-3.5) -- (0.5,-3.5);
		\draw[thick, red] (3.5,0) -- (3.5,-2.5);
		\draw[thick, red] (0,-2.5) -- (3.5,-2.5);

	\end{tikzpicture}
	\caption{$C \in \mathcal{C}_3((6,4,4,2,1))$ with $\inv(C) = 6$}
	\label{fig:rook-placement-cancelled}
\end{figure}
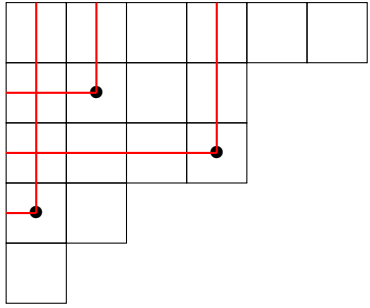

Since conjugation interchanges cells in the same column above with cells in the same row to the left of a given cell, \begin{equation}
	R_k(\lambda;q) = R_k(\lambda';q).
\end{equation} 

Because two rooks can not lie in the same row or in the same column, $R_k(\lambda;q) = 0$ for $k>\lambda_1$ or $k>\ell(\lambda)$.

\subsection{Recursions for $q$-rook numbers}\label{sec:recforqrook}For $\lambda = (\lambda_1,\lambda_2,\lambda_3,\ldots) \in \Par$, let $\widetilde{\lambda} = (\lambda_2,\lambda_3,\ldots)$ be the partition obtained by removing the first row. The  $q$-rook numbers $R_k(\lambda;q)$ for $0 \leq k \leq \lambda_1$ are determined by the recursions \cite[Theorem 1.1]{Garsia_Remmel_rook} \begin{align}\label{eq:GarsiaRemmelrec}
	R_k(\lambda;q) = q^{\lambda_1-k}R_k(\widetilde{\lambda};q) + [\lambda_1-k+1]_q R_{k-1}(\widetilde{\lambda};q), 
\end{align} with initial conditions $R_0(\lambda;q) = q^{|\lambda|}$.

\subsection{$q$-Stirling numbers}Let $n \in \ZZ_{\geq 0}$ and $\rho_n = (n-1,n-2,\ldots,1,0)$ be the staircase partition. Then \begin{equation}\label{eq:qStirling2}
R_{n-k}(\rho_n;q) = S_q(n,k) \qquad \hbox{ for } \ 0 \leq k \leq n
\end{equation} are the $q$-Stirling numbers of second kind \cite[(I.9)]{Garsia_Remmel_rook}. They satisfy the recursions $$ S_q(n,k) = q^{k-1}S_q(n-1,k-1) + [k]_q S_q(n-1,k) \hbox{ for } 0 \leq k \leq n,$$ and $S_q(0,0) = 1$, $S_q(n,k) = 0$ for $k<0$ or $k>n$.

\subsection{Rectangular $q$-rook numbers}
Let $a,b \in \ZZ_{\geq 0}$ and $0 \leq i \leq \min(a,b)$. \cite[Proposition 2.15]{CMP_rook} gives 
\begin{equation}\label{eq:rectangleqrook}
	R_i((b^a);q) = q^{(a-i)(b-i)} \dfrac{[a]_q!}{[a-i]_q!} {b \brack i}_q
\end{equation}

\subsection{$\ell(\lambda)$th $q$-rook numbers}

For $\lambda \in \Par$ with $\ell(\lambda) = \ell$,  \cite[Proposition 2.2]{CMP_rook} says
\begin{equation}\label{eq:Rlastprod}
	R_{\ell}(\lambda;q) = \prod_{i =1}^{\ell}[\lambda_{\ell-i+1}-i+1]_q.
\end{equation} 

\subsection{Tableaux formula for $q$-rook numbers}\label{sec:tabformulanots}

Let $\pi \in \DD_n$ for some $n \in \ZZ_{>0}$ and $\mu \vdash n$. Let $T \in \SYT^\pi_\mu$. Recall that \begin{align}
	\gamma(T) = \#\{(b,c) \in \mu\times \mu\,|\, \coleg(b)>\coleg(c) \hbox{ and } (T(c),T(b)) \in \Area(\pi) \},
\end{align} and for a box $b \in \mu$, let \begin{align}
	\gamma(T,b) = \#\{ c \in \mu\,|\, \coleg(b)>\coleg(c) \hbox{ and } (T(c),T(b)) \in \Area(\pi) \}, 
\end{align} then $$ \gamma(T) = \sum_{b \in \mu}\gamma(T,b).$$ For a box $b \in \mu$, denote by $\arm_{<_\pi j}(b)$ the number of boxes $c$ in the right of $b$ in the same row such that $T(c) <_\pi j$, i.e, \begin{align}
	\arm_{<_\pi j}(b) = \# \{ c \in \mu\,|\, \coleg(c) = \coleg(b) ,\, \mathrm{coarm}(c) > \mathrm{coarm}(b) ,\, T(c) <_\pi j \}.
\end{align} 

Let \begin{equation}\label{eq:wtdef}
	\wt(T;q) =  q^{n(\mu')-\# \Area(\pi) + \gamma(T)} \prod_{\substack{b \in \mu \\ \coleg(b)>0}} [\arm_{<_\pi T(b)}(\up(b))+1]_q,
\end{equation} where recall that $\# \Area(\pi)$ is the number of cells below $\pi$ strictly above the diagonal and $n(\mu') = \sum_{i} \binom{\mu_i}{2}$.

We restate \autoref{th:mainthmintrostate} from the introduction here.
\begin{theorem}\label{th:rooktableauformula}
	Let $\lambda \in \Par$ and $\pi \in \DD_n$ is such that $\lambda(\pi) = \lambda$. Then for $k \geq 0$, \begin{equation}\label{eq:R_k=sumwt}
		R_k(\lambda;q) = \sum_{\substack{\mu \vdash n \\ \mu_1 = n-k}} \sum_{T \in \SYT^\pi_\mu} \wt(T;q).
	\end{equation}
\end{theorem}

\subsection{Example} For the Dyck path in \autoref{fig:pi}, $n = 6$, and let $k = 3$. Then by \eqref{eq:SYTpieg}, the sum in \eqref{eq:R_k=sumwt} runs only over two tableaux, $$ T = \ytableaushort{123,456} \in \SYT^\pi_{(3,3)} \qquad\hbox{ and }\qquad S = \ytableaushort{123,45,6} \in \SYT^\pi_{(3,2,1)},$$ with $$ \gamma(T) = 0 \qquad\hbox{ and }\qquad \gamma(S) = 1,$$ and using $\#\Area(\pi) = 5$, the weights are \begin{align*}
	\wt(T;q) = q[2]_q[3]_q \qquad\hbox{ and }\qquad \wt(S;q) =  q^{-1} q [2]_q[3]_q = [2]_q[3]_q.
\end{align*} So, \begin{align*}
R_{3}((4,3,3);q) = q[2]_q[3]_q + [2]_q[3]_q = [3]_q[2]_q^2.
\end{align*}

\subsection{Proof of \autoref{th:rooktableauformula}}

We now prove \autoref{th:rooktableauformula} by showing that the right hand side of \eqref{eq:R_k=sumwt} satisfies the recursions \eqref{eq:GarsiaRemmelrec} for the $q$-rook numbers. 

\begin{lemma}\label{lem:inmax}
	Let $\pi \in \DD_n$. If $(i,n) \in \Area(\pi)$ then $i$ is a maximal element with respect to $<_\pi$ order, i.e, there is no $j \in [n]$ such that $i <_\pi j$.
\end{lemma}
\begin{proof}
	If $(i,k) \in \Area(\pi)$ then $(i,j) \in \Area(\pi)$ for all $j \in \{i+1,\ldots,k\}$. In particular, $(i,n) \in \Area(\pi)$ means that $(i,j) \in \Area(\pi)$ for all $j \in \{i+1,\ldots,n\}$.
\end{proof} 

\begin{lemma}\label{lem:leg0lem}
	Let $\pi \in \DD_n$ and $\mu \vdash n$. Let $T \in \SYT^\pi_\mu$. If $b \in \mu$ is such that $(T(b),n) \in \Area(\pi)$ then $\leg(b) = 0$.
\end{lemma}
\begin{proof}
	By \autoref{lem:inmax}, $T(b)$ is maximal with respect to $<_\pi$. Hence there can be no box below $T(b)$, so $\leg(b) = 0$. 
\end{proof}

Let $\pi \in \DD_n$ with $\lambda(\pi) = \lambda$ and $\pi' \in \DD_{n-1}$ be the path obtained by removing the first row of $\pi$, i.e, $\pi'$ is obtained from $\pi$ by removing the last occurence of $NE$ in $\pi$. For the path $\pi$ from \autoref{fig:pi}, the path $\pi'$ is shown in \autoref{fig:pi'}.

\begin{figure}[h]
	\begin{minipage}{0.45 \textwidth}
		\begin{tikzpicture}
			\NEpath{0,0}{5}{5}{1,1,1,0,0,0,1,1,0,0};
		\end{tikzpicture}
	\end{minipage}
	\hfill
	\begin{minipage}{0.45 \textwidth}
		\begin{align*}
			&\pi' = N^3E^3N^2E^2 \in \DD_5,
			\\
			&\Area(\pi') = \{(1,2),(1,3),(2,3),(4,5)\},
			\\
			&\lambda(\pi') = (3,3). 	
		\end{align*}
	\end{minipage}
	\caption{Removal of last occurence of $NE$ from $\pi$}
	\label{fig:pi'}
\end{figure}

Then $\lambda(\pi')$ is obtained by removing the first row of $\lambda(\pi)$, which we denote by $\widetilde{\lambda} = (\lambda_2,\ldots).$ If $T \in \SYT^\pi_\mu$ for some $\mu \vdash n$ then if we remove the box with entry $n$ from $T$ we obtain an element of $\SYT^{\pi'}_{\mu-\varepsilon_i}$, where $i$ is the row of the box with entry $n$ and $\mu-\varepsilon_i = (\mu_1,\ldots,\mu_{i-1},\mu_i-1,\mu_{i+1},\ldots)$.

Note that $$\#\Area(\pi) - \#\Area(\pi') = n-1-\lambda_1.$$

Let $T \in \SYT^{\pi'}_\nu$ for some $\nu \vdash n-1$. Let $T^{+n,i}$ be the tableau obtained by adding a box with entry $n$ in the $i$th row of $T$, if $\nu+\varepsilon_i \in \Par$ and $T^{+n,i}$ is a valid tableau in $\SYT^{\pi}_{\nu+\varepsilon_i}$.

\begin{lemma}
	Let $T \in \SYT^{\pi'}_\nu$ and $i$ is such that $T^{+n,i} \in \SYT^\pi_{\nu+\varepsilon_i}$. For $j \geq 1$, let \begin{align}\label{eq:Njdef}
		N(j) &= \# \{ b \in \nu\,|\, \coleg_\nu(b) = j-1 ,\, \leg_\nu(b) = 0 \hbox{ and } T(b) <_\pi n \}, 
	\end{align} and $N(0) = 0$.
	Then 
	\begin{align}\label{eq:wtcomparisongen}
		\dfrac{\wt(T^{+n,i};q)}{\wt(T;q)}=q^{\nu_1 - n + 1  + \lambda_1} \cdot q^{- (N(1)+\ldots + N(i-1))} [N(i-1)]_q.
	\end{align}
\end{lemma}

\begin{proof}

	Using \autoref{lem:leg0lem}, \begin{align}
		N(j) &= \# \{ b \hbox{ in row } j \hbox{ of } T \hbox{ with } \leg(b) = 0 \hbox{ and } T(b) <_\pi n \} 
		\nonumber
		\\
		&= \arm_{<_\pi n}((j,\nu_{j+1}+1))+1 
		\nonumber
		\\
		&= \nu_{j}-\nu_{j+1} - \#\{b \hbox{ in row } j \hbox{ of } T \hbox{ with } (T(b),n) \in \Area(\pi)\}.
		\label{eq:Njprecrel}
	\end{align} 
	Then 
	\begin{align*}
		\gamma(T^{+n,i}) - \gamma(T) 
		&= \gamma(T^{+n,i},(i,\nu_i+1)) 
		= \sum_{j=1}^{i-1}(\nu_j-\nu_{j+1} - N(j)) 
		\\
		&= \nu_1 - \nu_i -  (N(1)+\ldots + N(i-1)).
	\end{align*} 	
	Using 	\begin{align*}
		\# \Area(\pi) - \# \Area(\pi') = n - 1 - \lambda_1,
		\qquad \hbox{ and } \qquad
		n((\nu+\varepsilon_i)')-n(\nu') = \nu_i,
	\end{align*} then
	 \begin{align*}
		&\big(n((\nu+\varepsilon_i)')-\# \Area(\pi)+ \gamma(T^{+n,i}) \big)
		- \big( n(\nu')-\# \Area(\pi') + \gamma(T) \big)
		\\
		&= \nu_i - (n-1-\lambda_1) +  \gamma(T^{+n,i},(i,\nu_i+1))
		\\
		&= \nu_i - (n-1-\lambda_1) + \nu_1 - \nu_i -  (N(1)+\ldots + N(i-1))
		\\
		&= \nu_1 - (n-1-\lambda_1)  - (N(1)+\ldots+N(i-1)),
	\end{align*} 
	and \begin{align*}
		\dfrac{\prod\limits_{\substack{b \in \nu+\varepsilon_i \\ \coleg_{\nu+\varepsilon_i}(b)>0}} [\arm_{<_\pi T^{+n,i}(b)}(\up(b))+1]_q}{\prod\limits_{\substack{b \in \nu \\ \coleg_\nu(b)>0}} [\arm_{<_{\pi'} T(b)}(\up(b))+1]_q} &= [\arm_{<_\pi T^{+n,i}((i,\nu_i+1))}(\up((i,\nu_i+1)))+1]_q
		\\
		&= [\arm_{<_{\pi} n}((i-1,\nu_i+1))+1]_q = [N(i-1)]_q.
	\end{align*}
\end{proof}

\begin{lemma}\label{lem:wtcomparison}
	Let $T \in \SYT^{\pi'}_\nu$ where $\pi' \in \DD_{n-1}$ is obtained from $\pi \in \DD_n$ by deleting the rightmost occurence of $NE$. Then with the notations from above,
	\begin{align}\label{eq:wtcomparisoncases}
		\sum_{i>1} \dfrac{\wt(T^{+n,i};q)}{\wt(T;q)} = [\lambda_1-n+\nu_1+1]_q, \qquad \hbox{ and } \qquad \dfrac{\wt(T^{+n,1};q)}{\wt(T;q)} = q^{\nu_1 - (n-1-\lambda_1)}.
	\end{align} 
\end{lemma}

\begin{proof}
We use the same notation as before from \eqref{eq:Njdef}. Using \eqref{eq:Njprecrel}, 
	\begin{align*}
		\sum_{j \geq 1} N(j) &= \sum_{j \geq 1}(\nu_{j}-\nu_{j+1} - \#\{b \hbox{ in row } j \hbox{ of } T \hbox{ with } (T(b),n) \in \Area(\pi)\})
		\\
		&= \nu_1 - \#\{b \in \nu  \hbox{ with } (T(b),n) \in \Area(\pi)\} 
		\\
		&= \nu_1 - \#\{1 \leq j \leq n-1 \hbox{ with } (j,n) \in \Area(\pi)\}
		\\
		&= \nu_1 - (n-1) + \lambda_1.
	\end{align*} Then using \eqref{eq:wtcomparisongen},
	\begin{align*}
		\sum_{i>1} \dfrac{\wt(T^{+n,i};q)}{\wt(T;q)}
		&= q^{\nu_1 - n +1  + \lambda_1} \sum_{i>1}   q^{- (N(1)+\ldots + N(i-1))} [N(i-1)]_q
		\\
		&= q^{\nu_1 - n +1  + \lambda_1} \sum_{i>1}   q^{- (N(1)+\ldots + N(i-1))} \dfrac{1-q^{N(i-1)}}{1-q}
		\\
		&= q^{\nu_1 - n +1  + \lambda_1} \sum_{i>1} \dfrac{q^{- (N(1)+\ldots + N(i-1))} - q^{- (N(1)+\ldots + N(i-2))}}{1-q}
		\\
		&= q^{\nu_1 - n +1  + \lambda_1} \dfrac{q^{-\sum_{j \geq 1} N(j)}  - 1}{1-q}
		= q^{\nu_1 - n +1  + \lambda_1} \dfrac{q^{-\lambda_1 + n- \nu_1 -1}  - 1}{1-q}
		\\
		&= [\lambda_1-n+\nu_1+1]_q.
	\end{align*}
	This gives the first statement of \eqref{eq:wtcomparisoncases}.
	\eqref{eq:wtcomparisongen} for $i=1$ gives the second statement of \eqref{eq:wtcomparisoncases}.
\end{proof}

Now we can finish the proof of \autoref{th:rooktableauformula}.
	\ytableausetup{centertableaux}
	Denote by $R'_k(\lambda;q)$ the right hand side of \eqref{eq:R_k=sumwt}. To show that $R'_k(\lambda;q) = R_k(\lambda;q)$, we show that $R'_k(\lambda;q)$ satisfies the determining recursions from \S\ref{sec:recforqrook}.
	
	Let $\pi' \in \DD_{n-1}$ be the path obtained by removing the first row of $\pi$. Then  $\lambda(\pi') = \widetilde{\lambda}$ is obtained by removing the first row of $\lambda(\pi)$.  Then by \autoref{lem:wtcomparison}, \begin{align*}
		&R'_k(\lambda;q) 
		\\
		&= \sum_{\substack{\nu \vdash n-1 \\ \nu_1 = n-k-1}} q^{\nu_1 - (n-1-\lambda_1)} \sum_{T \in \SYT^{\pi'}_\nu} \wt(T;q) + \sum_{\substack{\nu \vdash n-1 \\ \nu_1 = n-k}} [\lambda_1-n+\nu_1+1]_q \sum_{T \in \SYT^{\pi'}_\nu} \wt(T:q)
		\\
		&=  q^{n-k-1 - (n-1-\lambda_1)} R'_{k}(\widetilde{\lambda};q) +  [\lambda_1-n+n-k+1]_q R'_{k-1}(\widetilde{\lambda};q)
		\\
		&= q^{\lambda_1-k} R'_{k}(\widetilde{\lambda};q) +  [\lambda_1-k+1]_q R'_{k-1}(\widetilde{\lambda};q),
	\end{align*} which matches with the Garsia-Remmel recursions \eqref{eq:GarsiaRemmelrec}. When $k=0$, the sum in the right hand side of \eqref{eq:R_k=sumwt} only runs over the partition $\mu = (n)$. There is only one tableau in $\SYT^\pi_{(n)}$, which is $ T = \ytableaushort{1\ldots n}$\,. Since $\# \Area(\pi) = \binom{n}{2} - |\lambda| = n((n)') - |\lambda|$, then $\wt(T;q) = q^{n((n)')-\# \Area(\pi)} = q^{|\lambda|}$. This proves that $R'_0(\lambda;q) = q^{|\lambda|}$. Thus $R'_k(\lambda;q)$ satisfies the recursions with the initial conditions, hence $R'_k(\lambda;q) = R_k(\lambda;q)$.

\section{Unicellular LLT functions}

In this section we recall the definition and some basic properties and examples of unicellular LLT functions, following the exposition in \cite{Carlsson-Mellit-Shuffle}. We follow standard notations and conventions regarding symmetric functions and plethysm. In particular, for a symmetric function $f$, $f[X]$ denotes $f(x_1,x_2,\ldots)$.

\subsection{Dyck path symmetric functions}

Let $\pi \in \DD_n$. For a word $w = (w_1,\ldots,w_n) \in \ZZ_{>0}^n$, let $$ \inv(\pi,w) = \#\{(i,j)\in \Area(\pi)\,|\, w_i>w_j\}.$$ The unicellular LLT symmetric function corresponding to $\pi$ is a symmetric function denoted $\chi_\pi(q)$, defined by $$ \chi_\pi(q)[X] = \sum_{w \in \ZZ_{>0}^n}^{} q^{\inv(\pi,w)} x_w,$$ where for a word $w$ as above, let $x_w = \prod_{i} x_{w_i}$. A proof of symmetry of $\chi_\pi(q)$ can be found in \cite[Proposition 3.2]{Carlsson-Mellit-Shuffle}. \cite[Remark 3.6]{Carlsson-Mellit-Shuffle} also explains the connection with the `usual' unicellular LLT functions.

The maximum value of $\inv(\pi,w)$ for $w\in \ZZ_{>0}^n$ is obtained when all boxes in $\Area(\pi)$ contributes $1$, in which case it equals to $\#\Area(\pi)$. This means the highest power of $q$ in $\chi_\pi(q)$ is $\#\Area(\pi)$. The reverse polynomial is denoted $\widetilde{\chi}_\pi(q)$, defined by \begin{equation}\label{eq:tildechidef}
	\widetilde{\chi}_\pi(q) = q^{\# \Area(\pi)} \chi_\pi(q^{-1}).
\end{equation} 

\subsection{Examples}
Let $n \in \ZZ_{>0}$. 
\begin{enumerate}
	\item For $\pi \in \DD_n$, $$ \chi_\pi(1) = h_1^n.$$
	\item Suppose $\pi \in \DD_n$ and $\pi$ does not touch the diagonal line from $(0,0)$ to $(n,n)$ except at the two ends. Then $$\chi_\pi(0)[X] = \sum_{\substack{w \in \ZZ_{>0}^n \\ w_1 \leq \ldots \leq w_n}} x_w = h_n[X].$$ Hence, if $\pi \in \DD$ meets the line $x=y$ at points $(\alpha_1+\ldots+\alpha_i,\alpha_1+\ldots+\alpha_i)$ for $i \in \ZZ_{>0}$ for some composition $\alpha$, then, $$ \chi_\pi(0) = h_\alpha.$$
	\item Let $\pi = (NE)^n \in \DD_n$. Then $\Area(\pi) = \emptyset$ and so $$ \chi_{(NE)^n}(q) = e_1^n.$$ 
	\item Let $\pi = N^nE^n \in \DD_n$. Then $\Area(\pi) = \{(i,j)\,|\, 1 \leq i < j \leq  n\}$ is the maximum possible. Then for any word $w \in \ZZ_{>0}^n$, $\inv(\pi,w) = \inv(w)$ is the usual number of inversions of the word, and since $\inv$ is Mahonian (\cite[Theorem 1.3]{HaglundqtCatbook}), $$ \chi_{N^nE^n}(q) = \sum_{\mu \vdash n} {n \brack \mu}_q m_\mu = W_{(n)}(q),$$ where the right hand side denotes the $q$-Whittaker functions.
	
	\item Let $\rev : \DD \to \DD$ be the map that a takes a Dyck path to its reverse, i.e, the path obtained by reading the Dyck path from right to left and interchanging the $N$ and $E$ steps. 
	\cite[Proposition 3.3]{Carlsson-Mellit-Shuffle} says
	\begin{equation}
		\chi_\pi(q) = \chi_{\rev(\pi)}(q).
	\end{equation}
	
	\item Suppose $\pi,\eta \in \DD$ and let $\pi \cdot \eta$ denote their concatenation, then $$ \chi_{\pi \cdot \eta}(q) = \chi_\pi(q) \cdot \chi_\eta(q).$$
\end{enumerate}

\subsection{chromatic quasisymmetric functions}

For a Dyck path $\pi \in \DD_n$, let $X_\pi(q)$ be the chromatic quasisymmetric function of the graph with vertex set $[n]$ and edge set $$ \bigg\{\{i,j\}\,|\, 1 \leq i<j \leq n \hbox{ and } (i,j) \in \Area(\pi)\bigg\}.$$ The chromatic quasisymmetric function $X_\pi(q)$ is in fact a symmetric function given by $$ X_\pi(q)[X] = \sum_{\substack{w \in \ZZ_{>0}^n \\ (i,j) \in \Area(\pi) \Rightarrow w_i \neq w_j}} q^{\inv(\pi,w)} x_w. $$
\cite[Proposition 3.5]{Carlsson-Mellit-Shuffle} says \begin{equation}\label{eq:chromaticLLTplethysmrel}
	\chi_\pi(q)[X] = (q-1)^n X_\pi(q)\bigg[ \dfrac{X}{q-1} \bigg].
\end{equation}

\subsection{$\omega$-involution}
\cite[Proposition 3.4]{Carlsson-Mellit-Shuffle} says that \begin{equation}\label{eq:omegachipi}
	\omega(\chi_\pi(q)) = q^{\# \Area(\pi)} \chi_\pi(q^{-1}) = \widetilde{\chi}_\pi(q).
\end{equation}

\section{$W$-expansion of unicellular LLT}

In this section, we connect our formula for $q$-rook numbers with the results of \cite{GMRWW_Macexpansion}, showing in \autoref{th:equalitywithVienna} that our tableaux weights are essentially the same as certain specializations of the tableaux weights from \cite{GMRWW_Macexpansion}, thus they give the coefficients of unicellular LLT functions in the $q$-Whittaker basis. Therefore, it follows in \autoref{prop:Rkwithcpimu} that the $q$-rook numbers are sum of $W$-coefficients (for a fixed first row length) of the unicellular LLTs.

\subsection{}Let $\widetilde{H}_\lambda(q,t)$ for $\lambda \in \Par$ be the modified Macdonald functions, with notations same as \cite{HHL-I}. The $q$-Whittaker functions are \cite{BergeronqWhittaker} $$ W_\lambda(q) = q^{n(\lambda')}\omega\widetilde{H}_{\lambda}(q^{-1},0) \qquad \hbox{ for } \lambda \in \Par.$$ 
Let \begin{equation}\label{eq:tildeWdef}
	\widetilde{W}_\lambda(q) = q^{n(\lambda')}W_\lambda(q^{-1}),
\end{equation}
then \begin{equation}\label{eq:omegatildeH=tildeW}
	\omega(\widetilde{H}_\lambda(q,0)) = \widetilde{W}_\lambda(q).
\end{equation} 
The first equation of \S4 of \cite{GMRWW_Macexpansion} says 
\begin{equation}\label{eq:QprimeplethysmtotildeH}
	Q_{\lambda'}(q^{-1})\bigg[\dfrac{X}{q-1}\bigg] = q^{-|\lambda|}q^{-n(\lambda')}\widetilde{H}_\lambda(q,0)[X],
\end{equation} where $Q_{\lambda'}(q)$ is the same as in \cite[Chapter III]{MacMainBook}.

\subsection{}For $\pi \in \DD_n$ and partitions $\mu \vdash n$, let $c_{\pi,\mu}(q) \in \QQ(q)$ be defined by \begin{align}\label{eq:cpimudef}
	\chi_\pi(q) = \sum_{\mu \vdash n} (1-q)^{n-\mu_1} c_{\pi,\mu}(q) W_\mu(q).
\end{align}
Let \begin{equation}\label{eq:tildectoc}
	\widetilde{c}_{\pi,\mu}(q) = q^{\# \Area(\pi) - n(\mu')}c_{\pi,\mu}(q^{-1}).
\end{equation} Recall from \eqref{eq:tildechidef} and \eqref{eq:tildeWdef} that $$ \widetilde{\chi}_\pi(q) = q^{\# \Area(\pi)} \chi_\pi(q^{-1}) \qquad \hbox{ and } \qquad	\widetilde{W}_\lambda(q) = q^{n(\lambda')}W_\lambda(q^{-1}).$$  Then \begin{equation}
\widetilde{\chi}_\pi(q) = \sum_{\mu \vdash n}^{} (1-q^{-1})^{n-\mu_1} \widetilde{c}_{\pi,\mu}(q) \widetilde{W}_\mu(q),
\end{equation}  
or applying $\omega$, using \eqref{eq:omegachipi} and \eqref{eq:omegatildeH=tildeW} and the fact that $\omega$ is an involution, \begin{equation}\label{eq:tildecastildeHcoeff}
	\chi_\pi(q) = \sum_{\mu \vdash n}^{} (1-q^{-1})^{n-\mu_1} \widetilde{c}_{\pi,\mu}(q) \widetilde{H}_\mu(q,0).
\end{equation} 

\begin{proposition}\label{th:equalitywithVienna} 
	For $\pi \in \DD_n$ and $\mu \vdash n$,
	\begin{align}
		c_{\pi,\mu}(q) 
		&= q^{-n(\mu')+\#\Area(\pi)} \cdot \sum_{T \in \SYT_\mu^\pi} \wt(T;q) 
		\nonumber
		\\
		&= \sum_{T \in \SYT^\pi_\mu}  
		q^{\gamma(T)} \prod_{\substack{b \in \mu \\ \coleg(b)>0}} [\arm_{<_\pi T(b)}(\up(b))+1]_q.
	\end{align} In particular, $c_{\pi,\mu}(q) \in \ZZ_{\geq 0}[q]$.
\end{proposition}

\begin{proof}
	The proposition is just a restatement of \cite[Theorem 4.1]{GMRWW_Macexpansion}. We now explain the changes required from their notations to our notation. Note that their $\chi_\pi$ is our $X_\pi$ and our $\chi_\pi$ is their $F_\pi$
	
	First, since $\chi_\pi(q) = \chi_{\rev(\pi)}(q)$, we have changed their order $\ll$ to our $<_\pi$, where in the notation of \cite{GMRWW_Macexpansion}, for $1 \leq i < j \leq n$, $i \ll j$ if $(i,j) \notin \Area(\rev(\pi))$ and in our notation $i<_\pi j$ if $(i,j) \in \Area(\pi)$.    
	
	In \cite{GMRWW_Macexpansion}, the authors use another partial order $\prec$ on $[n]$, defined by $i \prec j$ if $(i,j) \in \Area(\rev(\pi))$, which with our conventions then translate to $i \prec j$ if $(i,j) \in \Area(\pi)$.
	
	For a tableau $T \in \SYT^\pi_\mu$, let $T_1$ denote the fillings in the first row of $T$, for $1 \leq i \leq n$, let $T_{\prec i}$ be the skew shape with fillings $\prec i$, and $T_{<i}$ be the skew shape with fillings $<i$. Let $i$ appear in row $s$ of $T$ and $d(T,i) = \mathrm{sh}(T_{<i})_s - \mathrm{sh}(T_{<i})_{s+1}$, $L(T,i)$ is the coarm of the leftmost element in row $s$ of $T_{\prec i}$, $m(T,i)$ is the number of elements that are atleast $2$ rows above the box with filling $\prec i$.  
	
	Let $$ \widetilde{\wt}(T;q) = \prod_{i \notin T_1} q^{-m(T,i)-d(T,i)} \prod_{\substack{i \notin T_1 \\ (T_{\prec i})_s = \emptyset}} [d(T,i)]_q \prod_{i: (T_{\prec i})_s \neq \emptyset} [L(T,i) - \mathrm{sh}(T_{<i})_{s+1}]_q, $$ where the notations assume that $i$ appears in row $s+1$ of $T$. Then \cite[Theorem 4.1]{GMRWW_Macexpansion} says that \begin{align*}
		X_\pi(q) = \sum_{\mu \vdash n} \bigg(\sum_{T \in \SYT^\pi_\mu} \widetilde{\wt}(T;q)\bigg) \dfrac{q^{\binom{n+1}{2}-|\lambda(\pi)|}}{(q-1)^{\mu_1}} Q_{\mu'}(q^{-1}).
	\end{align*}
	Note that $|\lambda(\pi)| = \binom{n}{2} - \#\Area(\pi)$. Then using \eqref{eq:chromaticLLTplethysmrel} and \eqref{eq:QprimeplethysmtotildeH},
	\begin{align*}
		\chi_\pi(q) &= \sum_{\mu \vdash n} \bigg(\sum_{T \in \SYT^\pi_\mu} \widetilde{\wt}(T;q)\bigg) \dfrac{q^{n+\#\Area(\pi)}}{(q-1)^{\mu_1}} \cdot (q-1)^n Q_{\mu'}(q^{-1})\bigg[\dfrac{X}{q-1}\bigg]
		\\
		&= \sum_{\mu \vdash n} \bigg(\sum_{T \in \SYT^\pi_\mu} \widetilde{\wt}(T;q)\bigg) q^{n+\#\Area(\pi)}\cdot (q-1)^{n-\mu_1}Q_{\mu'}(q^{-1})\bigg[\dfrac{X}{q-1}\bigg]
		\\
		&= \sum_{\mu \vdash n} \bigg(\sum_{T \in \SYT^\pi_\mu} \widetilde{\wt}(T;q)\bigg) q^{n+\#\Area(\pi)}\cdot (q-1)^{n-\mu_1}q^{-n}q^{-n(\mu')}\widetilde{H}_\mu(q,0)[X] 
		\\
		&= \sum_{\mu \vdash n} \bigg(\sum_{T \in \SYT^\pi_\mu} \widetilde{\wt}(T;q)\bigg) q^{\#\Area(\pi)-n(\mu')} \cdot q^{n-\mu_1}(1-q^{-1})^{n-\mu_1}\widetilde{H}_{\mu}(q,0)[X].
	\end{align*}
Comparing with \eqref{eq:tildecastildeHcoeff},
\begin{align*}
	\widetilde{c}_{\pi,\mu}(q) = q^{\# \Area(\pi) - n(\mu')+ n - \mu_1} \sum_{T \in \SYT^\pi_\mu} \widetilde{\wt}(T;q),
\end{align*}
or, using \eqref{eq:tildectoc}, \begin{equation}\label{eq:castildewt}
	c_{\pi,\mu}(q) = q^{-n+\mu_1} \sum_{T \in \SYT^\pi_\mu} \widetilde{\wt}(T;q^{-1}).
\end{equation}

Suppose that $i$ appears in row $s+1$ in the box $T^{(i)}$. If $(T_{\prec i})_s = \emptyset$ then $d(T,i)$ is 1+the number of boxes $c$ in row $s$ in the arm of $\up(T^{(i)})$ whose value $T(c) < i$. Since $T(c) \nprec i$, this means $T(c) <_\pi i$. On the other hand, if $T(c) <_\pi i$, then $T(c) <i $ and $T(c) \nprec i$. So $$ d(T,i) = \arm_{<_\pi i}(\up(T^{(i)}))+1, \qquad\hbox{ if } (T_{\prec i})_s = \emptyset.$$

Suppose $(T_{\prec i})_{s} \neq \emptyset$. If $j \prec i$ then $(j,i) \in \Area(\pi)$ and if $k \geq j$, then $(k,i)$ has to be below the path as well, so $k \nless_\pi i$.    Then $L(T,i) - \mathrm{coarm}(T^{(i)})$ is 1+the number of boxes $c$ in the arm of $\up(T^{(i)})$ such that $T(c) <_\pi i$. So $$ L(T,i) - \mathrm{sh}(T_{<i})_{s+1} = \arm_{<_\pi i}(\up(T^{(i)}))+1, \qquad\hbox{ if } (T_{\prec i})_{s} \neq \emptyset.$$

Then $$ \widetilde{\wt}(T;q) = \prod_{i \notin T_1} q^{-m(T,i)-d(T,i)} \prod_{i \notin T_1} [\arm_{<_\pi i}(\up(T^{(i)}))+1 ]_q $$ 

Note that
\begin{align*}
	d(T,i) = \arm_{<i}(\up(T^{(i)}))+1
\end{align*}
and \begin{align*}
	\widetilde{\wt}(T;q^{-1}) &= 
	\prod_{i \notin T_1} q^{m(T,i)+d(T,i)} [\arm_{<_\pi i}(\up(T^{(i)}))+1]_{q^{-1}} 
	\\
	&= \prod_{i \notin T_1} q^{m(T,i)+\arm_{<i}(\up(T^{(i)}))+1-\arm_{<_\pi i}(\up(T^{(i)}))}[\arm_{<_\pi i}(\up(T^{(i)}))+1]_{q},
\end{align*}
and $\arm_{<i}(\up(T^{(i)}))-\arm_{<_\pi i}(\up(T^{(i)}))$ is the number of boxes in row $s$ with fillings $\prec i$. So $m(T,i)+\arm_{<i}(\up(T^{(i)}))-\arm_{<_\pi i}(\up(T^{(i)}))$ is the number of boxes with fillings $\prec i$ which are above $i$ (such a filling must occur to the right of $i$). Then $m(T,i)+\arm_{<i}(\up(T^{(i)}))-\arm_{<_\pi i}(\up(T^{(i)})) = \gamma(T,T^{(i)})$.

Then comparing with \eqref{eq:wtdef}, $$ \widetilde{\wt}(T;q^{-1}) = q^{n-\mu_1} \prod_{i\notin T_1} q^{\gamma(T,T^{(i)})} [\arm_{<_\pi i}(\up(T^{(i)}))+1]_{q} = q^{n-\mu_1-n(\mu')+\# \Area(\pi)} \wt(T;q). $$ Then by \eqref{eq:castildewt},  $$  c_{\pi,\mu}(q) = q^{-n(\mu')+\#\Area(\pi)} \cdot \sum_{T \in \SYT_\mu^\pi} \wt(T;q).$$

\end{proof}

\begin{corollary}\label{prop:Rkwithcpimu}
	Let $\pi \in \DD_n$ with $\lambda(\pi) = \lambda$. Suppose \begin{align*}
		\chi_\pi(q) = \sum_{\mu \vdash n} (1-q)^{n-\mu_1} c_{\pi,\mu}(q) W_\mu(q) \qquad \hbox{ and } \qquad \widetilde{\chi}_\pi(q) = \sum_{\mu \vdash n}^{} (1-q^{-1})^{n-\mu_1} \widetilde{c}_{\pi,\mu}(q) \widetilde{W}_\mu(q).
	\end{align*}
	Then
	\begin{align*}
R_k(\lambda;q) =		\sum_{\substack{\mu \vdash n \\ \mu_1 = n-k}} q^{n(\mu')-\# \Area(\pi)} c_{\pi,\mu}(q) \qquad \hbox{ and } \qquad R_k(\lambda;q^{-1}) = \sum_{\substack{\mu \vdash n \\ \mu_1 = n-k}} \widetilde{c}_{\pi,\mu}(q).
	\end{align*}
\end{corollary}
The recent paper \cite{kim2025halllittlewoodexpansionschromaticquasisymmetric} obtains another proof of \autoref{prop:Rkwithcpimu}.

\section{Abelian Dyck paths}

In this section we first provide a condition for which partitions appear in \eqref{eq:R_k=sumwt}. Then we focus our attention to Abelian Dyck paths, which are paths $\pi$ such that if $\lambda = \lambda(\pi)$ then $|\pi|\geq \lambda_1 + \lambda'_1$. We show that in the case of Abelian Dyck paths the sum in \eqref{eq:R_k=sumwt} only runs over a single partition. We then provide another proof of a result of Guay-Paquet that says in the case of Abelian Dyck paths, the unicellular LLT functions are sum of unicellular LLT functions with rectangle shapes, where the coefficients are given by certain $q$-hit numbers, which are closely related with the $q$-rook numbers.

\subsection{A condition for $\SYT^\pi_\mu \neq \emptyset$}
Recall that a subset of a poset $P$ is a chain if any two elements are comparable, and it is an anti-chain if any two distinct elements are incomporable.

For $\pi \in \DD_n$ let $P(\pi)\vdash n$ denote the Greene shape of the poset determined by $<_\pi$ on $[n]$, i.e, $P(\pi)_{1}+\ldots+P(\pi)_k$ is the maximum number of elements in a union of $k$ anti-chains in $[n]$ with respect to $<_\pi$. By \cite[Theorem 1.5]{Greene_somepartitions}, $P(\pi)'_1+\ldots+P(\pi)'_k$ is the maximum number of elements in a union of $k$ chains in $[n]$ 
with respect to $<_\pi$, where $P(\pi)'$ denotes the conjugate partition of $P(\pi)$.

For the path $\pi$ from \autoref{fig:pi}, $P(\pi) = (3,2,1)$. For example, the sets $\{1,2,3\}, \{4,5\}, \{6\}$ are antichains of length $3,2,1$ respectively, and the sets $\{1,4,6\}, \{2,5\}$ and $\{3\}$ are chains of length $3,2,1$ respectively.

\begin{lemma}\label{prop:piSYTsupport}
	Let $\pi \in \DD_n$ and $\mu \vdash n$ be such that $\SYT^\pi_\mu \neq \emptyset$. Then $\mu \geq P(\pi)$ in the dominance order.
\end{lemma}

\begin{proof}
	Since entries in each column increase according to $<_\pi$, the entries in each column is a chain in $[n]$, therefore $\mu'_1+\ldots+\mu'_k \leq$ the maximum number of elements in a union of $k$ chains with respect to $<_\pi$ $=P(\pi)_1'+\ldots+P(\pi)'_k$. So, $\mu' \leq P(\pi)'$, or $\mu \geq P(\pi)$ in dominance order.
\end{proof}

\subsection{Abelian Dyck paths}

For $n,m \in \ZZ_{\geq 0}$, denote by $(m^n)$ the rectangular partition $(m,\ldots,m)$ with $n$ rows with all parts equal to $m$. A partition $\mu \subseteq (m^n)$ if $\mu_1 \leq m$ and $\mu'_1 \leq n$. 

For $\mu \subseteq (m^n)$, denote by $\pi^{n,m}(\mu) \in \DD_{n+m}$ the path with $\lambda(\pi^{n,m}(\mu)) = \mu$.

\begin{proposition}\label{prop:abelianGreene}
	Let $\lambda \in \Par$, $\lambda \subseteq (m^n)$ and $\pi = \pi^{n,m}(\lambda)$. Then $P(\pi)'_1 \leq 2$.
\end{proposition}

\begin{proof}
	Suppose $1 \leq i<_\pi j<_\pi k \leq m+n$. Then $i<j<k$ and $(i,j) \notin \Area(\pi)$ and $(j,k) \notin \Area(\pi)$. Now, $(i,j) \notin \Area(\pi)$ implies that $j > m+n-\lambda'_i \geq m+n-n = m$, but then for any $k>j$, $(j,k) \in \Area(\pi)$, a contradiction. Then the maximum length of an chain in $([m+n],<_\pi)$ has to be $\leq 2$. Therefore, \autoref{prop:piSYTsupport} proves the claim.
\end{proof}

\begin{proposition}\label{prop:Rkabelian}
	Let $\lambda \in \Par$ and let $N \geq \lambda_1+\lambda'_1$. Suppose $\pi \in \DD_N$ is such that $\lambda(\pi) = \lambda$. Then \begin{align}\label{eq:Rkabelian}
		R_k(\lambda;q) &= q^{|\lambda|-(N-k)k} c_{\pi,(N-k,k)}(q). 
	\end{align}
\end{proposition}

\begin{proof}
	Taking $m \geq \lambda_1$ and $n \geq \lambda'_1$ such that $m+n = N$ in \autoref{prop:abelianGreene}, $P(\pi)'_1 \leq 2$ and if $\SYT^\pi_\mu \neq \emptyset$ then $\mu' \leq P(\pi)'$, thus $\mu'_1 \leq 2$. Then the summands in $R_k(\lambda;q)$ from \eqref{eq:R_k=sumwt} runs over $\mu \vdash N$ with $\mu_1 = N-k$, there is only one possibility of $\mu$, namely, $\mu = (N-k,k)$. 
	Using $\# \Area(\pi) = \binom{N}{2} - |\lambda|$, and \begin{align*}
		n((N-k,k)') - \#\Area(\pi) = \binom{N-k}{2} + \binom{k}{2} - \binom{N}{2} + |\lambda| = |\lambda| - (N-k)k
	\end{align*} in \autoref{prop:Rkwithcpimu} gives the statement.
	
\end{proof}

\subsection{Proof of \cite[Theorem 1.3]{CMP_rook}}
In this subsection we provide another proof of \cite[Theorem 1.3]{CMP_rook}, where it is attributed to Guay-Paquet's unpublished work. It says that for abelian Dyck paths, the unicellular LLT functions are a sum of the corresponding functions for rectangle shaped paths, with coefficients $q$-hit numbers.

Let $\lambda \subseteq (m^n)$ be a partition with $n \leq m$. Recall from \cite[Definition 2.3]{CMP_rook} the $q$-hit numbers of $\lambda$ are defined for $k \in \ZZ_{\geq 0}$, by \begin{align}
	H_k^{m,n}(\lambda;q) = \dfrac{q^{\binom{k}{2}-|\lambda|}}{[m-n]_q!} \sum_{i=k}^{n} R_i(\lambda;q) [m-i]_q! {i \brack k}_q (-1)^{i+k} q^{mi - \binom{i}{2}},
\end{align} and the reverse relation is \cite[(2.3)]{CMP_rook} \begin{equation}\label{eq:RkabwrtH}
	R_k(\lambda;q) = q^{|\lambda|-mk} \dfrac{[m-n]_q!}{[m-k]_q!} \sum_{j=k}^{n} H_j^{m,n}(\lambda;q) {j \brack k}_{q^{-1}}.
\end{equation}

\begin{proposition}[\kern-5pt{\sc\cite[Theorem 1.3]{CMP_rook}}]
	Let $\lambda \subseteq (m^n)$ with $n \leq m$. Let $\pi = \pi(\lambda) = \pi^{n,m}(\lambda) \in \DD_{n+m}$ be the Dyck path such that $\lambda(\pi) = \lambda$, and for $0 \leq j \leq n$, let $\pi(m^j) \in \DD_{n+m}$ be the Dyck paths for which $\lambda(\pi(m^j)) = (m^j)$. Then $$ \chi_{\pi(\lambda)}(q) = \dfrac{[m-n]_q!}{[m]_q!} \sum_{j=0}^{n} H^{m,n}_j(\lambda;q) \cdot \chi_{\pi(m^j)}(q).$$
\end{proposition}

Note that the version in \cite{CMP_rook} is about chromatic symmetric functions, which is equivalent to the statement above by using \eqref{eq:chromaticLLTplethysmrel}.

\begin{proof}
	
	Using \autoref{prop:Rkabelian} and \eqref{eq:rectangleqrook}, \begin{align}
		c_{\pi((m^j)),(m+n-k,k)}(q) &= q^{-mj+(m+n-k)k}R_k((m^j);q) 
		\nonumber
		\\
		&= q^{-mj+(m+n-k)k} \cdot q^{(j-k)(m-k)} \dfrac{[j]_q!}{[j-k]_q!} {m \brack k}_q
		\nonumber
		\\
		&= q^{(n-j)k} \dfrac{[j]_q!}{[j-k]_q!} {m \brack k}_q.
		\label{eq:cabrectangle}
	\end{align} In particular, $c_{\pi((m^j)),(m+n-k,k)} =0$ if $k>j$. By \eqref{eq:cpimudef} and \autoref{prop:abelianGreene}, \begin{align}\label{eq:chirectangle}
		&\chi_{\pi(m^j)}(q) = \sum_{k=0}^{j} (1-q)^k c_{\pi((m^j)),(m+n-k,k)}(q) W_{(m+n-k,k)}(q). 
	\end{align}
	Using \autoref{prop:Rkabelian}, \eqref{eq:RkabwrtH}, \eqref{eq:qbinomialreverse}, \eqref{eq:qbindef} and \eqref{eq:cabrectangle},	
	\begin{align*}
		c_{\pi(\lambda),(m+n-k,k)}(q) &= q^{-|\lambda|+(m+n-k)k}  R_k(\lambda;q) 
		\\
		&= q^{-|\lambda|+(m+n-k)k}\cdot q^{|\lambda|-mk} \dfrac{[m-n]_q!}{[m-k]_q!} \sum_{j=k}^{n} H_j^{m,n}(\lambda;q) {j \brack k}_{q^{-1}}  
		\\
		&= q^{(n-k)k}\cdot  \dfrac{[m-n]_q!}{[m-k]_q!} \sum_{j=k}^{n} H_j^{m,n}(\lambda;q) {j \brack k}_{q^{-1}}  
		\\
		&= q^{(n-k)k}\cdot  \dfrac{[m-n]_q!}{[m-k]_q!} \sum_{j=k}^{n} H_j^{m,n}(\lambda;q) q^{-k(j-k)} {j \brack k}_q
		\\
		&= q^{(n-k)k}\cdot  \dfrac{[m-n]_q!}{[m-k]_q!} \sum_{j=k}^{n} H_j^{m,n}(\lambda;q) q^{-k(j-k)} \dfrac{[j]_q!}{[k]_q! [j-k]_q!}
		\\
		&= q^{(n-k)k}\cdot  \dfrac{[m-n]_q!}{[m]_q!} \sum_{j=k}^{n} H_j^{m,n}(\lambda;q) q^{-k(j-k)} \dfrac{[j]_q!}{[k]_q! [j-k]_q!} \dfrac{[m]_q!}{[m-k]_q!}
		\\
		&= \dfrac{[m-n]_q!}{[m]_q!} \sum_{j=k}^{n} H_j^{m,n}(\lambda;q) q^{(n-j)k} \dfrac{[j]_q!}{[j-k]_q!} {m \brack k}_q
		\\
		&= \dfrac{[m-n]_q!}{[m]_q!} \sum_{j=k}^{n} H_j^{m,n}(\lambda;q) c_{\pi((m^j)),(m+n-k,k)}(q).
	\end{align*}
	Therefore, using \eqref{eq:chirectangle}, \begin{align*}
		&\chi_{\pi(\lambda)}(q) = \sum_{k=0}^{m+n} (1-q)^{k} c_{\pi(\lambda),(m+n-k,k)} W_{(m+n-k,k)}(q) \\
		&=  \sum_{k=0}^{m+n} (1-q)^{k} \bigg(\dfrac{[m-n]_q!}{[m]_q!} \sum_{j=k}^{n} H_j^{m,n}(\lambda;q) c_{\pi((m^j)),(m+n-k,k)}(q)\bigg) W_{(m+n-k,k)}(q)
		\\
		&=  \sum_{k=0}^{n} (1-q)^{k} \bigg(\dfrac{[m-n]_q!}{[m]_q!} \sum_{j=k}^{n} H_j^{m,n}(\lambda;q) c_{\pi((m^j)),(m+n-k,k)}(q)\bigg) W_{(m+n-k,k)}(q)
		\\
		&= \dfrac{[m-n]_q!}{[m]_q!} \sum_{j = 0}^{n} H_j^{m,n}(\lambda;q) \cdot \sum_{k=0}^{j} (1-q)^k c_{\pi(m^j),(m+n-k,k)}(q) W_{(m+n-k,k)}(q)
		\\
		&= \dfrac{[m-n]_q!}{[m]_q!} \sum_{j=0}^{n} H^{m,n}_j(\lambda;q) \cdot \chi_{\pi(m^j)}(q).
	\end{align*}
	
\end{proof}

\section{$n$th rook numbers from $e$-coefficients}

In this section, we show that if $\lambda \subseteq (n^n)$ is a partition with $n-i+1\leq\lambda_{i}\leq n$ for every $i \in [n]$, then the $n$-th $q$-rook number $R_n(\lambda;q)$ can be obtained from taking sums of coefficients from the $e$-expansion of certain unicellular LLT functions. 

For $\pi \in \DD_n$ and $\mu \vdash n$, define $b_{\pi,\mu}(q) \in \QQ(q)$ by \begin{equation}
	\chi_\pi(q) = \sum_{\mu \vdash n} (q-1)^{n-\ell(\mu)} b_{\pi,\mu}(q) e_\mu.
\end{equation}

A formula for $b_{\pi,\mu}(q)$ is obtained in \cite[Theorem 1.3]{Abreu_Nigro_forests}, where it is also shown that $b_{\pi,\mu}(q) \in \ZZ_{\geq 0}[q]$ for all $\pi,\mu$. We do not need the precise formulas here. The following is our main result in this section.

\begin{proposition}\label{prop:enrook}
	Let $\pi \in \DD_n$ with $\lambda(\pi) = \lambda$ and $\lambda^{c}=\lambda(\pi)^{c}=(n-\lambda_{n},n-\lambda_{n-1},...,n-\lambda_{1})$ denotes the complimentary partition of $\lambda$ inside $(n^n)$. Then \begin{equation}\label{eq:enrook}
		\sum_{\mu\vdash n}q^{n-\ell(\mu)}b_{\pi,\mu}(q)=\prod_{j=1}^{n}[n-\lambda_{j}-j+1]_{q}=R_{n}(\lambda^{c};q).
	\end{equation}
\end{proposition}
\eqref{eq:Rlastprod} is the second equality above.
To show the first equality, we prove that both sides of \eqref{eq:enrook} are multiplicative and satisfy the modular laws of Abreu and Nigro from \cite{Abreu_Nigro_csfmodular}, \cite{Abreu_Nigro_forests}. \cite[Theorem 1.2]{Abreu_Nigro_csfmodular} says that such functions are completely determined by their values on the paths $N^nE^n$ for $n \in \ZZ_{\geq 0}$, and we show that two sides of the first equality in \eqref{eq:enrook} are equal for these paths.

For $\pi,\eta \in \DD$, let $\pi \cdot \eta$ denote the concatenation of two Dyck paths. For an algebra $A$, a function $f: \DD \to A$ is multiplicative if $f(\pi \cdot \eta) = f(\pi) \cdot f(\eta)$ for any $\pi,\eta \in \DD$.

Taking the generators $y_n = q^{-1}p_n$ for $n \in \ZZ_{>0}$ of the ring of symmetric functions in \cite[Definition 3.1]{Abreu_Nigro_forests} we get $$ \mathrm{IF}(\pi) = \sum_{\mu \vdash n} b_{\pi,\mu}(q) q^{-\ell(\mu)}p_\mu \qquad \hbox{ for } \pi \in \DD_n.$$ Taking specialization at $(1,0,\ldots)$, and multiplying by $q^{|\pi|}$, \begin{equation}
	q^n\mathrm{IF}(\pi)[1] = \sum_{\mu \vdash n}  q^{n-\ell(\mu)} b_{\pi,\mu}(q) \qquad \hbox{ for } \pi \in \DD_n, 
\end{equation} which is the left hand side of \eqref{eq:enrook}.
By \cite[Proposition 3.3, 3.4]{Abreu_Nigro_forests} $\pi \mapsto \mathrm{IF}(\pi)$ is multiplicative and satisfy the modular laws. Then so is $\pi \mapsto q^n\mathrm{IF}(\pi)[1]$. 

\autoref{prop:R_nsquare} says that the two sides of the first equality in \eqref{eq:enrook} agree on paths $N^nE^n$ for $n \in \ZZ_{>0}$ and \autoref{prop:partoprodmodular} and the above discussion says that both of them are multiplicative and satisfy modular law. Hence their equality is proved by \cite[Theorem 1.2]{Abreu_Nigro_csfmodular}.

We now provide the details concerning the product side of \eqref{eq:enrook}.

A similar result has been proved in \cite[Proposition 3.8]{Abreu_Nigro_forests}, which says for $\pi \in \DD_n$ and $\lambda = \lambda(\pi)$, $$ \sum_{\mu \vdash n}^{} b_{\pi,\mu}(q) = \prod_{i=1}^{n} (1+[n-\lambda_j-j]_q).$$ 

\subsection{$q$-Stirling numbers of first kind}

The $q$-Stirling numbers of first kind $s_{q}(n,k)$ are defined by \cite[page 4]{Abreu_Nigro_forests}
\begin{equation*}\label{eq:qStirling1def}
	x(x-[1]_{q})...(x-[n-1]_{q})=\sum_{k=1}^{n}(-1)^{n-k}s_{q}(n,k)x^{k}.
\end{equation*}
Putting $x=-z^{-1}$ and multiplying by $(-z)^n$, we get \begin{equation}\label{eq:qStirling1genfn}
		\sum_{k=1}^{n} s_q(n,k)z^{n-k} = \prod_{i=0}^{n-1}(1+[i]_qz).
	\end{equation}
	In other words, $$s_q(n,k) = e_{n-k}([0]_q,\ldots,[n-1]_q).$$

\subsection{Equality for $N^nE^n$}
Now we prove \eqref{eq:enrook}  when $\pi = N^nE^n$.
\begin{lemma}\label{prop:R_nsquare} For $n\in \ZZ_{>0}$,
	\begin{equation}\label{eq:Rnnn}
		\sum_{\lambda\vdash n}q^{n-\ell(\lambda)}b_{N^{n}E^{n},\lambda} = [n]_{q}! = R_n((n^n);q)
	\end{equation}
\end{lemma}
\begin{proof}
	By \eqref{eq:Rlastprod}, $$ R_n((n^n);q) = [n]_q!.$$
	\cite[Corollary 3.7]{Abreu_Nigro_forests} says that 
	$$\sum_{\substack{\lambda\vdash n\\ \ell(\lambda)=k}}b_{N^{n}E^{n},\lambda}=s_{q}(n,k).$$
	Then by \eqref{eq:qStirling1genfn},
	$$\sum_{\lambda\vdash n}q^{n-\ell(\lambda)}b_{N^{n}E^{n},\lambda}=\sum_{k=1}^{n}q^{n-k}s_{q}(n,k)=[n]_{q}!.$$ 
\end{proof}

\subsection{Multiplicative and modular}
\begin{lemma}\label{prop:partoprodmodular}
	Define $G: \DD \to \ZZ[q]$ by $$ G(\pi) =  \prod_{j=1}^{n}[n-\lambda(\pi)_{j}-j+1]_{q} \hbox{ for } \pi \in \DD_n.$$ Then $G$ is multiplicative and satisfies the modular law.
\end{lemma}

\begin{proof}
	For showing that $G$ is multiplicative, let $\pi \in \DD_n,\eta \in \DD_m$ and $\pi\cdot\eta$ denotes the concatenation of those two Dyck paths. Then $$ \lambda(\pi \cdot \eta)_j = \begin{cases}
		\lambda(\pi)_j &\hbox{ for } 1 \leq j \leq n,
		\\
		\lambda(\eta)_{j-n}+n &\hbox{ for } n+1 \leq j \leq m.
	\end{cases}$$ 
	Then
	\begin{align*}
		&G(\pi)\cdot G(\eta)=
		\prod_{i=1}^{n}[n-\lambda(\pi)_{i}-i+1]_{q}.\prod_{j=1}^{m}[m-\lambda(\eta)_{j}-j+1]_{q}\\
		=&\prod_{i=1}^{n}[n+m-\lambda(\pi)_{i}-(i+m)+1]_{q}.\prod_{j=1}^{m}[m+n-(\lambda(\eta)_{j}+n)-j+1]_{q}\\
		=&\prod_{j=1}^{m+n}[n+m-\lambda(\pi\cdot \eta)_{j}-j+1]_{q} = G(\pi \cdot \eta).
	\end{align*}
	Hence, $G$ is multiplicative. 
	
	Let $\pi^0, \pi^1. \pi^2$ be Dyck paths satifying conditions $(1)$ or $(2)$ of modular law of \cite[Definition 2.1]{Abreu_Nigro_csfmodular} with $\lambda(\pi^{(i)}) = \lambda^{(i)}$ for $i \in \{0,1,2\}$. Using the transformation given by \S\ref{sec:DycktoHess}, one of the following is true \begin{enumerate}
		\item[(1)] $\lambda^{(0)} = \lambda^{(1)} + \varepsilon_s, \hbox{ and } \lambda^{(2)} = \lambda^{(1)}-\varepsilon_s,$ for some $s\in \ZZ_{>0}$,
		\item[or,] 
		\item[(2)] $
		\lambda^{(2)} = \lambda^{(1)} - \varepsilon_r,\, \lambda^{(0)} = \lambda^{(1)} + \varepsilon_{r+1},\, \hbox{ and } \lambda^{(1)}_r - \lambda^{(1)}_{r+1} = 1$ for some $r \in \ZZ_{>0}$,
	\end{enumerate} 
To show that modular law holds we need to show that for any three Dyck paths $\pi^0,\pi^1,\pi^2$ with associated partitions $\lambda^{(0)},\lambda^{(1)},\lambda^{(2)}$ satisfying $(1)$ or $(2)$ as above, 
$$(1+q)G(\pi^{(1)})=qG(\pi^{(0)})+G(\pi^{(2)}).$$ 

Both of these cases are proved by using \begin{equation}\label{eq:modularbasic}
	(1+q)[a]_{q} = q[a-1]_{q} + [a+1]_{q}, \qquad \hbox{ for } a \in \ZZ_{\geq 0}.
\end{equation} $$ 	$$

Assume that condition $(1)$ is true. In this case, since the difference is only on one component $s$, taking $$a = n - \lambda^{(1)}_s - s +1$$ gives $$ n - \lambda^{(0)}_{s} - s +1 = a-1, \hbox{ and } n - \lambda^{(2)}_{s} - s +1 = a+1,$$ so using \eqref{eq:modularbasic} shows that modular law holds true in this case.

Next, assume that condition $(2)$ is true. Let $$a = n-\lambda^{(1)}_{r}-r+1 = n - \lambda^{(1)}_{r+1} - (r+1)+1 \hbox{ and } b = \prod_{j\neq r,r+1}^{}[n-\lambda^{(1)}_{j}-j+1]_{q}.$$ Then using \begin{align*}
	&n-\lambda^{(2)}_r-r+1 = n-(\lambda^{(1)}_r -1)-r+1 = a+1,
	\\
	& n-\lambda^{(2)}_{r+1}-(r+1)+1 = n-\lambda^{(1)}_{r+1}-r = n - \lambda^{(1)}_r - r +1 = a,
	\\
	& n-\lambda^{(0)}_{r}-r+1 = n-\lambda^{(1)}_{r}-r+1 = a,
	\\
	&n-\lambda^{(0)}_{r+1}-(r+1)+1 = n-(\lambda^{(1)}_{r+1}+1)-(r+1)+1 = a-1,
\end{align*} we get,
\begin{align*}
	&qG(\pi^{(0)})+G(\pi^{(2)}) = b \cdot ([a+1]_{q}[a]_{q}+q[a]_{q}[a-1]_{q}) =b\cdot (1+q)[a]_{q}[a]_{q}
	\\
	&=\prod_{j\neq r,r+1}^{}[n-\lambda^{(1)}_{j}-j+1]_{q}\cdot(1+q)[n-\lambda^{(1)}_{r}-r+1]_{q}[n-\lambda^{(1)}_{r+1}-(r+1)+1]_{q}
	\\
	&= (1+q) G(\pi^{(1)}).
\end{align*}
Hence, $G$ satisfies the modular law.

\end{proof}

\section{Further Comments}

\subsection{A formula for the $q$-Stirling numbers}

Recall from \eqref{eq:qStirling2} that the $q$-Stirling numbers of second kind are $S_q(n,k) = R_{n-k}(\rho_n;q)$, where $\rho_n = (n-1,\ldots,0)$ is the staircase partition. \autoref{th:rooktableauformula} in this case becomes a sum over the usual standard Young tableaux.

\begin{proposition}
	Let $n,k \in \ZZ_{>0}$, the $q$-Stirling numbers of second kind has the formula $$S_q(n,k) = \sum_{\substack{\mu \vdash n \\ \mu_1 = k}} q^{n(\mu')} \sum_{T \in \SYT_\mu} \prod_{\substack{b \in \mu \\ \coleg(b)>0}} [\arm_{<T(b)}(\up(b)+1)]_q.$$ 
\end{proposition}

\begin{proof}
	For $n \in \ZZ_{>0}$, let $\lambda = \rho_n = (n-1,n-2,\ldots,1,0)$. Let $\pi = \pi(\lambda) \in \DD_n$. Then $\# \Area(\pi) = 0$ and $i,j \in [n]$, $i <_\pi j$ if and only if $i<j$. So for $\mu \vdash n$, $\SYT^\pi_\mu = \SYT_\mu$ and for $T \in \SYT_\mu$ and $b \in \mu$, $\gamma(T,b) = 0$. Then $$ \wt(T;q) =   q^{n(\mu')} 
	\prod_{\substack{b \in \mu \\ \coleg(b)>0}} [\arm_{< T(b)}(\up(b))+1]_q,$$ so \autoref{th:rooktableauformula} gives the result.
	
\end{proof}

\subsection{Matrix counting over $\mathbb{F}_q$}

Let $\pi \in \DD_n$ and $\lambda = \lambda(\pi)$, $P_k(\pi;q)$ be the number of $n \times n$ matrices over $\mathbb{F}_q$ of rank $k$ such that all non-zero entries appear above $\pi$. Then by \cite[Theorem 1]{Haglund_qrookfinitefield} $$ P_k(\pi;q) = (q-1)^k q^{|\lambda|-k} R_k(\lambda;q^{-1}) = (1-q^{-1})^k q^{|\lambda|} R_k(\lambda;q^{-1}). $$ Then using \autoref{prop:Rkwithcpimu}, \begin{equation}\label{eq:numHessmat}
	P_k(\pi;q) = q^{|\lambda|}(1-q^{-1})^k \sum_{\substack{\mu \vdash n \\ \mu_1 = n-k}} \widetilde{c}_{\pi,\mu}(q) = q^{|\lambda|}\sum\limits_{\substack{\mu \vdash n \\ \mu_1 = n-k}} [\widetilde{W}_\mu(q)] \widetilde{\chi}_\pi(q),
\end{equation} where $[\widetilde{W}_\mu(q)] \widetilde{\chi}_\pi(q)$ denotes the coefficient of $\widetilde{W}_\mu(q)$ in the $\widetilde{W}$-expansion of $\widetilde{\chi}_\pi(q)$. 

Since $\sum_{k \geq 0} P_k(\pi;q)$ is the total number of matrices such that all non-zero entries lie above $\pi$, which is simply $q^{|\lambda|}$, we get $$ \sum_{\mu \vdash n} [\widetilde{W}_\mu(q)] \widetilde{\chi}_\pi(q) = 1.$$

\bibliographystyle{alpha}
\bibliography{pathsym}

\newcommand{\etalchar}[1]{$^{#1}$}
\begin{thebibliography}{GMR{\etalchar{+}}25}

\bibitem[AN21a]{Abreu_Nigro_csfmodular}
Alex Abreu and Antonio Nigro.
\newblock Chromatic symmetric functions from the modular law.
\newblock {\em J. Comb. Theory, Ser. A}, 180:31, 2021.
\newblock Id/No 105407.

\bibitem[AN21b]{Abreu_Nigro_forests}
Alex Abreu and Antonio Nigro.
\newblock A symmetric function of increasing forests.
\newblock {\em Forum Math. Sigma}, 9:21, 2021.
\newblock Id/No e35.

\bibitem[Ber20]{BergeronqWhittaker}
F.~Bergeron.
\newblock A {Survey} of $q$-{Whittaker} polynomials.
\newblock Preprint, {arXiv}:2006.12591, 2020.

\bibitem[CM18]{Carlsson-Mellit-Shuffle}
Erik Carlsson and Anton Mellit.
\newblock A proof of the shuffle conjecture.
\newblock {\em J. Am. Math. Soc.}, 31(3):661--697, 2018.

\bibitem[CMP23]{CMP_rook}
Laura Colmenarejo, Alejandro~H. Morales, and Greta Panova.
\newblock Chromatic symmetric functions of {Dyck} paths and {{\(q\)}}-rook
  theory.
\newblock {\em Eur. J. Comb.}, 107:36, 2023.
\newblock Id/No 103595.

\bibitem[GMR{\etalchar{+}}25]{GMRWW_Macexpansion}
Sean~T. Griffin, Anton Mellit, Marino Romero, Kevin Weigl, and Joshua~Jeishing
  Wen.
\newblock On {Macdonald} expansions of $q$-chromatic symmetric functions and
  the {Stanley}-{Stembridge} {Conjecture}.
\newblock Preprint, {arXiv}:2504.06936, 2025.

\bibitem[GR86]{Garsia_Remmel_rook}
A.~M. Garsia and J.~B. Remmel.
\newblock Q-counting rook configurations and a formula of {Frobenius}.
\newblock {\em J. Comb. Theory, Ser. A}, 41:246--275, 1986.

\bibitem[Gre76]{Greene_somepartitions}
Curtis Greene.
\newblock Some partitions associated with a partially ordered set.
\newblock {\em J. Comb. Theory, Ser. A}, 20:69--79, 1976.

\bibitem[Hag98]{Haglund_qrookfinitefield}
James Haglund.
\newblock {{\(q\)}}-rook polynomials and matrices over finite fields.
\newblock {\em Adv. Appl. Math.}, 20(4):450--487, 1998.

\bibitem[Hag08]{HaglundqtCatbook}
James Haglund.
\newblock {\em The {{\(q,t\)}}-{Catalan} numbers and the space of diagonal
  harmonics. {With} an appendix on the combinatorics of {Macdonald}
  polynomials}, volume~41 of {\em Univ. Lect. Ser.}
\newblock Providence, RI: American Mathematical Society (AMS), 2008.

\bibitem[HHL05]{HHL-I}
J.~Haglund, M.~Haiman, and N.~Loehr.
\newblock A combinatorial formula for {M}acdonald polynomials.
\newblock {\em J. Amer. Math. Soc.}, 18(3):735--761, 2005.

\bibitem[KLY25]{kim2025halllittlewoodexpansionschromaticquasisymmetric}
Jang~Soo Kim, Seung~Jin Lee, and Meesue Yoo.
\newblock Hall--littlewood expansions of chromatic quasisymmetric polynomials
  using linked rook placements.
\newblock Preprint, {arXiv}:2506.23082, 2025.

\bibitem[Mac95]{MacMainBook}
I.~G. Macdonald.
\newblock {\em Symmetric functions and {H}all polynomials}.
\newblock Oxford Mathematical Monographs. The Clarendon Press, Oxford
  University Press, New York, second edition, 1995.
\newblock With contributions by A. Zelevinsky, Oxford Science Publications.

\bibitem[RS23]{Ram_Schlosser_qWandrook}
Samrith Ram and Michael~J. Schlosser.
\newblock Diagonal operators, $q$-{Whittaker} functions and rook theory.
\newblock Preprint, {arXiv}:2309.06401, 2023.

\end{thebibliography}

\end{document}